\documentclass{amsart}

\usepackage[left = 3 cm, right = 3 cm,
]{geometry}

\usepackage{
	amssymb,
	amsfonts,
	amsmath,
	mathtools,
	amscd,
	amsthm,
	commath,
	esdiff,
	cases
}

\usepackage{hyperref}
\usepackage{url}

\newtheorem{theorem}{Theorem}[section]
\newtheorem{lemma}[theorem]{Lemma}
\newtheorem{proposition}[theorem]{Proposition}
\newtheorem{corollary}[theorem]{Corollary}

\newtheorem*{theorem*}{Theorem}

\newtheorem*{lemma*}{Lemma}

\theoremstyle{definition}

\numberwithin{equation}{section}
\numberwithin{figure}{section}

\begin{document}
	
	\title[On umbilical real hypersurfaces of products of complex space forms]{On umbilical real hypersurfaces of products of\\ complex space forms}
	
	\author[I. Domingos]{Iury Domingos}
	\author[R. da Silva]{Ranilze da Silva}
	\author[A. de Sousa]{Alexandre de Sousa}
	\author[F. Vitório]{Feliciano Vitório}
	
	\address{Universidade Federal de Alagoas\\
		Av. Manoel Severino Barbosa S/N,
		57309-005 Arapiraca - AL, Brazil}
	\email{iury.domingos@arapiraca.ufal.br}

	\address{Universidade Federal de Alagoas\\
		Instituto de Matem\'{a}tica\\
		Campus A. C. Sim\~{o}es, BR 104 - Norte, Km 97, 57072-970, Macei\'o - AL, Brazil}
	\email{maria.ranilze@im.ufal.br}
	\email{feliciano@pos.mat.ufal.br}

	\address{Secretaria de Educa\c{c}\~{a}o do Estado do Cear\'{a}\\
		EEMTI Maria Thom\'{a}sia \\
		Rua Pol\^{o}nia 369, Maraponga, 60710-500, Fortaleza – CE, Brazil}
	\email{alexandre.mota3@prof.ce.gov.br}

	\thanks{This work was partially financed by the Coordena\c{c}\~{a}o
		de Aperfei\c{c}oamento de Pessoal de N\'{i}vel Superior - Brasil (CAPES) -
		Finance Code 001. I.~Domingos was partially supported by the Brazilian National Council for Scientific and Technological Development (CNPq), grant no.~409513/2023-7.}

	
	
	\begin{abstract}
		Tashiro and Tachibana proved that there exist no totally umbilical hypersurfaces in complex space forms with nonzero constant holomorphic sectional curvature, and it is also known that the shape operator of such hypersurfaces cannot be parallel.
		Motivated by these results, we study real hypersurfaces in products of complex space forms.
		We establish rigidity and nonexistence results for totally umbilical real hypersurfaces in this setting.
		In particular, we show that if a real hypersurface in a product of complex space forms does not admit a local product structure, then its shape operator cannot be parallel.
		Moreover, we provide a classification of totally umbilical real hypersurfaces, showing that those admitting a local almost product structure are necessarily totally geodesic or extrinsic hyperspheres.
	\end{abstract}
	
	\maketitle
	
	\section{Introduction}
	
	Real space forms and their submanifolds have been extensively studied by many researchers. There are also several works devoted to the study of products of two real space forms.
	B.~Daniel~\cite{Benoit} provided necessary and sufficient conditions for a Riemannian manifold to be isometrically immersed into the products $\mathbb{S}^n \times \mathbb{R}$ or $\mathbb{H}^n \times \mathbb{R}$.
	Kowalczyk~\cite{Kowalcsyk} extended Daniel's results to products of two space forms. Moreover, Lira, Tojeiro, and Vitório~\cite{L-T-V} proved an existence and uniqueness theorem for isometric immersions of semi-Riemannian manifolds into products of semi-Riemannian space forms.
	
	Concerning umbilical hypersurfaces, Souam and Van der Veken~\cite{Souan-VdV} established existence conditions for totally umbilical hypersurfaces in Riemannian products of the form $M^n \times \mathbb{R}$ and provided a complete description of such hypersurfaces. 
	Mendonça and Tojeiro~\cite{Mendonça-Tojeiro} classified umbilical submanifolds of arbitrary codimension, extending the classification obtained by Souam and Van der Veken in $\mathbb{S}^n \times \mathbb{R}$.  In the product of two $2$-dimensional space forms, Nakad and Roth~\cite{Nakad-Roth} gave a characterization of totally umbilical hypersurfaces. More recently, de Lima and dos Santos~\cite{Lima-Santos} characterized nontrivial totally umbilical hypersurfaces in product spaces $M \times I$ and in warped products $I \times_\omega M$, showing that they arise locally as graphs over isoparametric families of totally umbilical hypersurfaces of $M$, and extending classification results to $\mathbb{S}^n \times \mathbb{R}$ and $\mathbb{H}^n \times \mathbb{R}$.
	
	In complex space forms, the works of Niebergall and Ryan, Liu and Xiao, and Yano and Kon~\cite{Niegerball-Ryan, Liu-Xiao, Yano-kon} provide fundamental material for the study of hypersurfaces. In particular, Niebergall and Ryan~\cite{Niegerball-Ryan} presented a detailed construction of important examples in the complex projective space and the complex hyperbolic space.
	They proved a theorem (Theorem~\ref{t4}) stating that the shape operator of a hypersurface in a complex space form with constant holomorphic sectional curvature cannot be parallel. They also showed that there are no totally umbilical hypersurfaces in $\mathbb{CP}^n$ or $\mathbb{CH}^n$.
	This latter fact was first established by Tashiro and Tachibana~\cite{Tashiro-Tachibana} in 1963.
	
	Motivated by these works, we investigate real hypersurfaces, with particular emphasis on umbilical hypersurfaces, in products of two complex space forms, assuming that at least one factor has nonzero holomorphic sectional curvature. In Section~\ref{chapter1}, we present definitions, notation, and basic properties of complex manifolds that are useful to understand this work. In Section~\ref{chapter2}, we study products of complex space forms, deriving their curvature tensor and the fundamental equations of a hypersurface.
	In Section~\ref{chapter3}, we obtain a rigidity result for the shape operator (Theorem~\ref{t1}).
	More precisely, we prove that if a real hypersurface is not $F$-invariant (Proposition~\ref{f-invariancia}), in particular, if it does not admit the induced almost product Riemannian structure, then its shape operator cannot be parallel.
	We then investigate totally umbilical real hypersurfaces. In the $F$-invariant case, we show that such hypersurfaces necessarily have constant mean curvature and are therefore either totally geodesic or extrinsic hyperspheres (Theorem~\ref{t3}).
	On the other hand, if the hypersurface fails to be $F$-invariant at some point, we prove that the mean curvature cannot be constant, and consequently no such hypersurface can be totally geodesic or an extrinsic hypersphere (Theorem~\ref{t2}).
	
	\subsection*{Acknowledgements} This work was initiated while Alexandre de Sousa was a CAPES fellow at the Institute of Mathematics of the Federal University of Alagoas,
	whose members he would like to thank for their hospitality.
	
	
	\section{Preliminaries}\label{chapter1}
	
	This section is devoted to a brief introduction to complex manifolds and to recalling notation, definitions, and properties that are used throughout the text. We place particular emphasis on complex space forms, which constitute the main object of study of this work.
	
	\subsection{Complex manifolds}
	
	Let $M$ be an $m$-dimensional differentiable manifold. The manifold $M$ is said to be \textit{almost complex} if there exists a differentiable bundle map $J \colon TM \rightarrow TM$ such that $J^2 = -I$. The map $J$ is called an \textit{almost complex structure} on $M$. Observe that if $M$ admits an almost complex structure, then $(\det J)^2 = (-1)^m$, which implies that the real dimension $m$ of $M$ must be even.
	
	An almost complex manifold $M$ is called a \textit{K\"ahler manifold} if $J$ is compatible with the metric and satisfies $\nabla J = 0$. That is, for all $X,Y \in TM$, we have
	\begin{equation*}
		\langle X,Y \rangle = \langle JX,JY \rangle
	\end{equation*}
	and
	\begin{equation*}
		(\nabla_X J)Y = \nabla_X(JY) - J\nabla_X Y = 0.
	\end{equation*}
	
	On K\"ahler manifolds, one defines the holomorphic sectional curvature.
	
	The \emph{holomorphic sectional curvature} is the sectional curvature computed on holomorphic planes, that is, on planes spanned by vectors of the form $\{X, JX\}$. More precisely,
	\begin{equation*}
		K(X,JX) = \dfrac{\langle R(X,JX)X,JX\rangle}{|X|^2|JX|^2 - \langle X,JX\rangle^2}.
	\end{equation*}
	
	A K\"ahler manifold is said to be a \emph{complex space form} if it has constant holomorphic sectional curvature equal to $16c$. In this case, we denote by $\mathbb{CQ}^n_k$ a complex space form of complex dimension $n$ and constant holomorphic sectional curvature $k = 16c$. We emphasize that the factor $16$ above is merely a matter of convention.
	
	The curvature tensor of a complex space form with holomorphic sectional curvature $16c$ is given by
	\begin{equation*}
		R(X,Y)Z = 4c\bigl(X \wedge Y + JX \wedge JY + 2\langle X,JY\rangle J\bigr)Z.
	\end{equation*}
	When interacting with the almost complex structure $J$, the curvature tensor of a complex space form satisfies the following properties:
	\begin{itemize}
		\item[1.] $R(X,Y) = R(JX,JY)$,
		\item[2.] $R(X,JY) = -R(JX,Y)$,
		\item[3.] $R(X,Y)J = JR(X,Y)$,
		\item[4.] $\langle R(X,Y)JZ, JT\rangle = \langle R(X,Y)Z,T\rangle$,
		\item[5.] $\langle R(X,Y)JZ, T\rangle = -\langle R(X,Y)Z,JT\rangle$.
	\end{itemize}
	
	The complex space forms are $\mathbb{C}^n$, the complex Euclidean space when $c = 0$, $\mathbb{CP}^n$, the complex projective space when $c > 0$, and $\mathbb{CH}^n$, the complex hyperbolic space when $c < 0$.
	
	The \textit{complex Euclidean space} $\mathbb{C}^n$ is endowed with the Euclidean metric
	\[
	\langle X,Y\rangle = \operatorname{Re}\!\left(\sum_{i=1}^n x_i \overline{y}_i\right),
	\]
	where $X = (x_1,\dots,x_n)$ and $Y = (y_1,\dots,y_n)$ belong to $\mathbb{C}^n$. The \textit{complex projective space} $\mathbb{CP}^n$ is defined by
	\[
	\mathbb{CP}^n = (\mathbb{C}^{n+1} \setminus \{0\})/\{p \sim \lambda p \,;\, \lambda \in \mathbb{C} \setminus \{0\}\},
	\]
	and is equipped with the Fubini--Study metric, which we now briefly describe.
	
	Let $\pi \colon \mathbb{C}^{n+1} \setminus \{0\} \rightarrow \mathbb{CP}^n$ be the natural projection, and consider its restriction to the unit sphere $\mathbb{S}^{2n+1}(1)$, namely,
	\[
	\pi \colon \mathbb{S}^{2n+1}(1) \subset \mathbb{C}^{n+1} \setminus \{0\} \rightarrow \mathbb{CP}^n.
	\]
	This restriction is surjective, and two points $p,q \in \mathbb{S}^{2n+1}(1)$ have the same image if and only if they lie on the same great circle, that is, $p = e^{it}q$ for some $t \in \mathbb{R}$. We endow $\mathbb{C}^{n+1} \setminus \{0\}$ with the Euclidean metric.
	Note that the position vector $P$, restricted to $\mathbb{S}^{2n+1}(1)$, is normal to $\mathbb{S}^{2n+1}(1)$, and the vector field $\eta = iP$ defines a unit vector field tangent to $\mathbb{S}^{2n+1}(1)$. One verifies that $(d\pi)_p$ is surjective and has kernel $\operatorname{span}\{\eta_p\}$, where $\eta_p = ip$, for any $p \in \mathbb{S}^{2n+1}(1)$.
	Thus, for any vector field $X \in T\mathbb{CP}^n$, there exists a unique vector field $\overline{X} \in T\mathbb{S}^{2n+1}(1)$ such that $(d\pi)\overline{X} = X$ and $\overline{X}$ is orthogonal to $\eta$. The vector field $\overline{X}$ is called the \textit{horizontal lift} of $X$, and the Fubini--Study metric on $\mathbb{CP}^n$ is defined by
	\[
	\langle X,Y \rangle_{\mathbb{CP}^n} = \langle \overline{X},\overline{Y} \rangle_{\mathbb{S}^{2n+1}}.
	\]
	
	For the \textit{complex hyperbolic space} $\mathbb{CH}^n$, we consider $\mathbb{C}^{n+1}_1 = (\mathbb{C}^{n+1},\langle \cdot,\cdot \rangle)$, where the metric is given by
	\[
	\langle X,Y\rangle = \operatorname{Re}\!\left(-x_0\overline{y}_0 + \sum_{i=1}^n x_i\overline{y}_i\right).
	\]
	We then define
	\[
	H_1^{2n+1} = \{p \in \mathbb{C}^{n+1}_1 \,;\, \langle p,p\rangle = -1\}.
	\]
	Using the same reasoning as above, with $H_1^{2n+1}$ in place of $\mathbb{S}^{2n+1}(1)$, we define $\mathbb{CH}^n$ as the space of equivalence classes of $H_1^{2n+1}$ under the action $p \mapsto \lambda p$. We thus obtain the projection $\pi \colon H_1^{2n+1} \rightarrow \mathbb{CH}^n$ and define its metric in the same manner as for the complex projective space.

	
	\section{Real hypersurfaces in products of complex space forms}\label{chapter2}
	
	\subsection{Products of complex space forms}
	
	Let $\mathbb{CQ}^{n_{1}}_{k_1}$ and $\mathbb{CQ}^{n_{2}}_{k_2}$ be Riemannian manifolds with constant holomorphic sectional curvatures $k_1 = 16c_1$ and $k_2 = 16c_2$, endowed with complex structures $J_1$ and $J_2$, respectively. We consider the product Riemannian manifold
	\[
	\overline{M} = \mathbb{CQ}^{n_{1}}_{k_1} \times \mathbb{CQ}^{n_{2}}_{k_2},
	\]
	equipped with the product metric. We denote by $\pi_i \colon \overline{M} \rightarrow \mathbb{CQ}^{n_i}_{k_i}$ the natural projection onto $\mathbb{CQ}_i := \mathbb{CQ}^{n_i}_{k_i}$, for $i \in \{1,2\}$.
	
	On $\overline{M}$, we consider the complex structure defined by $J = (J_1,J_2)$ and the almost product structure given by the endomorphism $F \colon T\overline{M} \rightarrow T\overline{M}$ defined by
	\[
	F = \pi_1 - \pi_2.
	\]
	It is customary to write $F = \pi_1 - \pi_2$ as shorthand for the pair $(\pi_1,-\pi_2)$.
	
	With this notation, we have $F \neq I$, $F^2 = I$.  So it follows that
	\[
	\langle FX, Y \rangle = \langle X, FY \rangle
	\Longleftrightarrow \langle FX, FY \rangle = \langle X, Y \rangle,
	\quad \forall X,Y \in \Gamma(T\overline{M}),
	\]
	where $I$ denotes the identity map on $T\overline{M}$.
	
	In order to study the curvature tensor of $\overline{M}$, we introduce the auxiliary operators
	\[
	\overline{L}_i = I + \varepsilon_i F \colon T\overline{M} \to T\overline{M},
	\]
	with $\varepsilon_1 = 1$ and $\varepsilon_2 = -1$. Using the properties of the complex structure $J$ and the almost product structure $F$, we obtain the following expression for the curvature tensor of the product of two complex space forms.
	
	\begin{proposition}\label{prop-curvatura}
		The curvature tensor $\overline{\mathcal{R}} \colon T\overline{M} \times T\overline{M} \times T\overline{M} \rightarrow T\overline{M}$ of $\overline{M} = \mathbb{CQ}_1 \times \mathbb{CQ}_2$ is given by
		\begin{equation*}
			\overline{\mathcal{R}}(\overline{X},\overline{Y})\overline{Z}
			= \sum_{i=1}^2 \frac{c_i}{2} \left[
			\overline{L}_i\overline{X} \wedge \overline{L}_i\overline{Y}
			+ J\overline{L}_i\overline{X} \wedge J\overline{L}_i\overline{Y}
			+ 2\langle \overline{L}_i\overline{X}, J\overline{L}_i\overline{Y} \rangle J
			\right]\overline{L}_i\overline{Z},
		\end{equation*}
		where $\overline{X},\overline{Y},\overline{Z} \in T\overline{M}$ and $\overline{L}_i = I + \varepsilon_i F$, with $\varepsilon_1 = 1$ and $\varepsilon_2 = -1$.
	\end{proposition}
	
	\begin{proof}
		Let $\overline{X} \in T\overline{M}$ and denote by $\overline{X}_i := \pi_i \overline{X}$ its projection onto $\mathbb{CQ}_i$. Let $\overline{R}_i$ be the curvature tensor of $\mathbb{CQ}_i$. For $\overline{X}, \overline{Y}, \overline{Z} \in T\overline{M}$, we have
		\begin{align*}
			\overline{\mathcal{R}}(\overline{X},\overline{Y})\overline{Z}
			&= \overline{R}_1(\overline{X}_1,\overline{Y}_1)\overline{Z}_1
			+ \overline{R}_2(\overline{X}_2,\overline{Y}_2)\overline{Z}_2 \\
			&= 4c_1\bigl(\overline{X}_1 \wedge \overline{Y}_1
			+ J_1\overline{X}_1 \wedge J_1\overline{Y}_1
			+ 2\langle \overline{X}_1, J_1\overline{Y}_1 \rangle J_1\bigr)\overline{Z}_1 \\
			&\quad + 4c_2\bigl(\overline{X}_2 \wedge \overline{Y}_2
			+ J_2\overline{X}_2 \wedge J_2\overline{Y}_2
			+ 2\langle \overline{X}_2, J_2\overline{Y}_2 \rangle J_2\bigr)\overline{Z}_2.
		\end{align*}
		
		By definition of the operators $\overline{L}_1$ and $\overline{L}_2$, we obtain
		\[
		\overline{L}_1 \overline{X} = 2\overline{X}_1,
		\qquad
		\overline{L}_2 \overline{X} = 2\overline{X}_2.
		\]
		Moreover, since $FJ = JF$, we have
		\[
		J_1\overline{L}_1\overline{X} = J\overline{L}_1\overline{X},
		\qquad
		J_2\overline{L}_2\overline{X} = J\overline{L}_2\overline{X}.
		\]
		Substituting these expressions into the previous formula we get our assertion.
	\end{proof}

	\subsection{Real hypersurfaces and their fundamental equations}
	
	Let $\overline{M} = \mathbb{CQ}_1 \times \mathbb{CQ}_2$ be endowed with its Levi-Civita connection $\overline{\nabla}$. Let $M$ be an oriented real hypersurface of $\overline{M}$. The Levi-Civita connection $\nabla$ of the induced metric on $M$ and the shape operator $A$ are characterized by
	\[
	\overline{\nabla}_X Y = \nabla_X Y + \langle AX, Y \rangle \nu,
	\qquad
	AX = -\overline{\nabla}_X \nu,
	\]
	where $\nu$ denotes the unit normal vector field along $M$. Here, the mean curvature function $H$ is defined as the normalized trace of the shape operator, namely,
	\[
	H = \frac{1}{2n-1} \operatorname{tr} A.
	\]
	
	The almost product structure $F$ induces on $M$ a vector field $V \in \Gamma(TM)$, a smooth function $h \colon M \rightarrow \mathbb{R}$, and an endomorphism $f \colon TM \rightarrow TM$ such that, for all $X \in \Gamma(TM)$,
	\[
	FX = fX + \langle V, X \rangle \nu,
	\qquad
	F\nu = V + h\nu.
	\]
	
	For the complex structure, we have
	\[
	JX = \varphi X + \langle W, X \rangle \nu,
	\]
	where $W := -J\nu$ is the structure vector field and $\varphi \colon TM \rightarrow TM$ is skew-symmetric. From the definitions of $J$ and $F$, it follows that $FJ = JF$. One easily verifies that, for all $X \in \Gamma(TM)$,
	\begin{equation}\label{eq2.2}
		\langle V,W \rangle = 0,
		\quad
		\varphi^2 X = -X + \langle X, W \rangle W,
		\quad
		\langle W,W \rangle = 1,
		\quad
		\varphi W = 0.
	\end{equation}
	
	We also recall the following properties (see~\cite{Nakad-Roth}). Let $M$ be a real hypersurface of $\overline{M} = \mathbb{CQ}_1 \times \mathbb{CQ}_2$ and let $X \in TM$. Then:
	\begin{equation}
		\left\{
		\begin{aligned}
			& f \text{ is symmetric}, \\
			& fV = -hV, \\
			& h^2 + |V|^2 = 1, \\
			& f\varphi X + \langle W, X \rangle V = \varphi fX - \langle V, X \rangle W, \\
			& fW = hW - \varphi V.
		\end{aligned}
		\right.
		\label{eq13}
	\end{equation}
	From the decomposition of $F$, it follows that
	\[
	f^2 X = X - \langle V, X \rangle V,
	\ \ \ \
	\langle fX, Y \rangle = \langle X, fY \rangle,
	\ \ \ \
	\langle fX, fY \rangle = \langle X, Y \rangle - \langle V, X \rangle \langle V, Y \rangle.
	\]
	
	Consequently, we get the following characterization of the $F$-invariant subsets of the real hypersurfaces (see~\cite[pg. 17]{SilvaUmainvestigacaohipersuperficies2023} and~\cite[pg. 424]{Yano-kon}).
	\begin{proposition}
		\label{f-invariancia}
		Let (M,g) be an oriented real hypersurface in a product of complex space forms.
		If  \(U \subseteq M\) is a nonempty subset, then the following statements are equivalent:
		\begin{itemize}
			\item[1.] \(U\) is $F$-invariant, that is, \(F(T_pM) \subset T_pM, \quad \forall p \in U\);
			\item[2.] \(V\vert_U \equiv 0\);
			\item[3.] \((f, g)\) is an almost product Riemannian structure on \(U\),
			that is, \(f^2 = I\) and \(f^{\ast}g = g\) everywhere on U.
		\end{itemize}
	\end{proposition}
	
	The relation between the connections $\overline{\nabla}$ and $\nabla$ yields the Gauss and Codazzi equations, which correspond, respectively, to the tangential and normal components of the curvature tensor.
	
	\begin{proposition}
		The Gauss and Codazzi equations are given by:
		\begin{align}
			R(X,Y)Z &= \sum_{i=1}^2 \frac{c_i}{2} \Big\{ (L_iX \wedge L_iY)L_iZ
			+ (\varphi L_iX \wedge \varphi L_iY)L_iZ \nonumber \\
			&\quad + \langle V,Z \rangle \big[ \langle V,Y \rangle L_iX - \langle V,X \rangle L_iY \nonumber \\
			&\qquad + \langle L_iY, W \rangle (\varphi L_iX - \langle V,X \rangle W)
			- \langle L_iX, W \rangle (\varphi L_iY - \langle V,Y \rangle W) \big] \nonumber \\
			&\quad + \big[(\langle V,X \rangle)\varphi L_iY - \langle V,Y \rangle \varphi L_iX\big] \wedge W \, L_iZ \nonumber \\
			&\quad + 2\big(\langle L_iX, \varphi L_iY - \langle V,Y \rangle W \rangle
			+ \varepsilon_i \langle L_iY, W \rangle \langle V,X \rangle\big)
			(\varphi L_iZ - \langle V,Z \rangle W) \Big\} \nonumber \\
			&\quad + (AX \wedge AY)Z, \nonumber
		\end{align}
		and
		\begin{align}\label{eq-codazzi}
			d_\nabla A(Y,X)
			&= (\nabla_X A)Y - (\nabla_Y A)X \nonumber \\
			&= \sum_{i=1}^2 \frac{c_i}{2} \Big[
			2\varepsilon_i L_i((Y \wedge X)V)
			+ L_i \varphi L_i((Y \wedge X)L_iW) \nonumber \\
			&\quad + (3\varepsilon_i \langle (Y \wedge X)V, L_iW \rangle
			+ 2\langle X, L_i \varphi L_iY \rangle)L_iW
			\Big],
		\end{align}
		where $L_i = I + \varepsilon_i f$, with $\varepsilon_1 = 1$ and $\varepsilon_2 = -1$.
	\end{proposition}
	
	\begin{proof}
		The proof follows from the properties of the decompositions of the complex structure and the almost product structure given in equations~\eqref{eq2.2} and~\eqref{eq13}.
	\end{proof}
	
	We conclude this section with a technical lemma that will be used in the next section.
	
	\begin{lemma}\label{l1}
		Let $\{e_j\}_{j=1}^{n-1}$ be a local orthonormal frame of $W^\perp$. Then
		\[
		\{W, e_j, \varphi e_j\}_{j=1}^{n-1}
		\]
		forms a local orthonormal frame. Moreover, the following identities hold:
		\begin{eqnarray}
			d_\nabla A(W, e_j)&=&\sum_{i=1}^{2} (4c_i(\varepsilon _i+h) \left \langle e_j,V \right \rangle W +\sum_{k} \frac{c_i}{2} \{[ (7+\varepsilon _ih)\left \langle e_j, V \right \rangle \left \langle  V,\varphi e_k\right \rangle \nonumber\\ &&+\left ( 1-\varepsilon _ih \right )\left \langle\varphi e_j,  V \right \rangle \left \langle V,e_k \right \rangle -(1+\varepsilon _ih) \left \langle L_i\varphi L_ie_j,e_k  \right \rangle ] e_k \nonumber\\
			&& +  [- (7+\varepsilon _ih)\left \langle e_j, V \right \rangle \left \langle  V,e_k\right \rangle +\left ( 1-\varepsilon _ih \right )\left \langle\varphi e_j,  V \right \rangle \left \langle \varphi V,e_k \right \rangle \nonumber \\
			&&+(1+\varepsilon _ih)\left \langle\varphi L_i\varphi L_ie_j,e_k  \right \rangle ] \varphi e_k \}), \nonumber
		\end{eqnarray}
		
		\begin{eqnarray}
			d_\nabla A(W,\varphi e_j)&=&\sum_{i=1}^{2} (4c_i (\varepsilon _i+h)\left \langle  V,\varphi e_j \right \rangle W +\sum_{k}    \frac{c_i}{2} \{[ (7+\varepsilon _ih)\left \langle V,\varphi e_j \right \rangle \left \langle V,\varphi e_k\right \rangle \nonumber \\ &&-\left (1-\varepsilon _ih \right )\left \langle V,e_j \right \rangle \left \langle V,e_k \right \rangle -(1+\varepsilon _ih)\left \langle L_i\varphi L_i\varphi e_j,e_k  \right \rangle ] e_k \nonumber\\
			&&+ [- (7+\varepsilon _ih)\left \langle \varphi e_j, V \right \rangle \left \langle  V,e_k\right \rangle -\left ( 1-\varepsilon _ih \right )\left \langle e_j,  V \right \rangle \left \langle  V,\varphi e_k \right \rangle \nonumber \\
			&&-(1+\varepsilon _ih)\left \langle L_i\varphi L_i\varphi e_j,\varphi e_k  \right \rangle ] \varphi e_k \})\nonumber
		\end{eqnarray}
		and
		\begin{eqnarray}
			d_\nabla A(e_j,\varphi e_l)&=&\sum_{i=1}^{2}\{ c_i[ (3+\varepsilon _ih)(\left \langle V,\varphi e_l \right \rangle \left \langle V,\varphi e_j\right \rangle +\left \langle V,e_l \right \rangle \left \langle V,e_j\right \rangle \nonumber \\
			&& -(1+\varepsilon _ih)\left \langle L_i\varphi L_i\varphi e_l,e_j  \right \rangle ] W\nonumber\\
			&&+ \sum_k \frac{c_i}{2}[2\varepsilon_i(\langle V,\varphi e_l\rangle \langle L_ie_j,e_k\rangle - \langle V,e_j\rangle \langle L_i\varphi e_l, e_k\rangle) \nonumber\\
			&&-\varepsilon_i (\langle V, e_l\rangle \langle L_i\varphi L_i e_j,e_k\rangle +\langle V,\varphi e_j\rangle \langle L_i\varphi L_i\varphi e_l,e_k \rangle) \nonumber\\
			&& +\varepsilon_i \langle V,\varphi e_k \rangle\left( 3(\langle V,\varphi e_j\rangle \langle V,\varphi e_l\rangle + \langle V, e_j\rangle \langle V,e_l\rangle)-2 \langle L_i \varphi L_i \varphi e_l, e_j \rangle \right) ]e_k\nonumber\\
			&&+ \sum_k\frac{c_i}{2}[2\varepsilon_i(\langle V,\varphi e_l\rangle \langle L_ie_j,\varphi e_k\rangle - \langle V,e_j\rangle \langle L_i\varphi e_l,\varphi e_k\rangle) \nonumber\\
			&&-\varepsilon_i (\langle V, e_l\rangle \langle L_i\varphi L_i e_j,\varphi e_k\rangle +\langle V,\varphi e_j\rangle \langle L_i\varphi L_i\varphi e_l,\varphi e_k \rangle) \nonumber\\
			&& -\varepsilon_i \langle V, e_k \rangle\left( 3(\langle V,\varphi e_j\rangle \langle V,\varphi e_l\rangle + \langle V, e_j\rangle \langle V,e_l\rangle)-2 \langle L_i \varphi L_i \varphi e_l, e_j \rangle \right) ]\varphi e_k\nonumber\}.
		\end{eqnarray}
	\end{lemma}

	\begin{proof}
		We may write
		\begin{equation}\label{eq8}
			\begin{array}{rcl}
				d_\nabla A(W, e_j)&=&\left \langle d_\nabla A(W, e_j), W \right \rangle W \\
				& & +\sum_{k} \left ( \left \langle d_\nabla A(W, e_j), e_k \right \rangle e_k
				+ \left \langle d_\nabla A(W, e_j),\varphi e_k \right \rangle \varphi e_k\right  )\\
				d_\nabla A(W,\varphi e_j)&=&\left \langle d_\nabla A(W, \varphi e_j), W \right \rangle W\\
				& & +\sum_{k} \left ( \left \langle d_\nabla A(W, \varphi e_j), e_k \right \rangle e_k
				+ \left \langle d_\nabla A(W, \varphi e_j),\varphi e_k \right \rangle \varphi e_k\right  )\\
				d_\nabla A(e_j,\varphi e_l)&=&\left \langle d_\nabla A(e_j, \varphi e_l), W \right \rangle W \\
				& & +\sum_{k} \left ( \left \langle d_\nabla A(e_j, \varphi e_l), e_k \right \rangle e_k
				+ \left \langle d_\nabla A(e_j, \varphi e_l),\varphi e_k \right \rangle \varphi e_k\right  ).
			\end{array}
		\end{equation}
		
		Using equations \eqref{eq2.2}, \eqref{eq13}, and the Codazzi equation~\eqref{eq-codazzi} and noting that
		\begin{eqnarray}
			\langle L_i W,\varphi e_k\rangle &=& -\varepsilon_i \langle V, e_k\rangle, \label{eq15}\\
			\langle L_i W, e_k\rangle &=& \varepsilon_i \langle V,\varphi e_k\rangle, \label{eq16}
		\end{eqnarray}
		the Codazzi equation yields
		\begin{eqnarray*}
			\left \langle d_\nabla A(W, e_j), W \right \rangle
			&=& \sum_{i=1}^{2}4c_i(\varepsilon _i+h) \left \langle e_j,V \right \rangle.
		\end{eqnarray*}
		
		Proceeding analogously, we compute
		\begin{eqnarray*}
			\left \langle d_\nabla A(W, e_j), e_k \right \rangle
			&=&\sum_{i=1}^{2}\frac{c_i}{2} \Big[ (7+\varepsilon _ih)\left \langle e_j, V \right \rangle \left \langle  V,\varphi e_k\right \rangle \nonumber \\
			&&+\left ( 1-\varepsilon _ih \right )\left \langle\varphi e_j,  V \right \rangle \left \langle V,e_k \right \rangle
			-(1+\varepsilon _ih)\left \langle L_i\varphi L_ie_j,e_k  \right \rangle \Big],
		\end{eqnarray*}
		and
		\begin{eqnarray*}
			\left \langle d_\nabla A(W, e_j),\varphi e_k \right \rangle
			&=&\sum_{i=1}^{2}\frac{c_i}{2} \Big[ - (7+\varepsilon _ih)\left \langle e_j, V \right \rangle \left \langle  V,e_k\right \rangle \nonumber \\
			&&+\left ( 1-\varepsilon _ih \right )\left \langle\varphi e_j,  V \right \rangle \left \langle \varphi V,e_k \right \rangle
			+(1+\varepsilon _ih)\left \langle\varphi L_i\varphi L_ie_j,e_k  \right \rangle \Big].
		\end{eqnarray*}
		
		Substituting these expressions into the first equation of \eqref{eq8}, we obtain the first identity of the lemma.
		
		By the Codazzi equation, we also derive
		\begin{eqnarray*}
			\left \langle d_\nabla A(W,\varphi e_j),W \right \rangle
			&=&\sum_{i=1}^{2} 4c_i (\varepsilon _i+h)\left \langle  V,\varphi e_j \right \rangle,
		\end{eqnarray*}
		together with the corresponding tangential components. Substituting them into \eqref{eq8} yields the second identity.
		
		Finally, using again the Codazzi equation, we obtain
		\begin{eqnarray*}
			d_\nabla A(e_j, \varphi e_l)
			&=& \sum_{i=1}^2  \dfrac{c_i}{2} [2\varepsilon_i L_i((e_j\wedge \varphi e_l)V)
			+ L_i\varphi L_i((e_j\wedge \varphi e_l)L_iW)\nonumber \\
			&&+(3\varepsilon_i \langle (e_j\wedge \varphi e_l)V, L_iW \rangle
			+2\langle \varphi e_l, L_i\varphi L_ie_j \rangle)L_iW].
		\end{eqnarray*}
		
		Using identities \eqref{eq2.2}, \eqref{eq13}, \eqref{eq15}, and \eqref{eq16}, we compute
		\begin{eqnarray*}
			\langle d_\nabla A(e_j, \varphi e_l),W\rangle
			&=&\sum_{i=1}^2 c_i\Big[(3+\varepsilon_ih)(\langle V,\varphi e_l\rangle \langle V,\varphi e_j\rangle
			+\langle V, e_l\rangle \langle V, e_j\rangle)\nonumber\\
			&&- (1+\varepsilon_ih)\langle L_i\varphi L_i\varphi e_l,e_j\rangle\Big].
		\end{eqnarray*}
		
		Substituting these expressions into the last equation of \eqref{eq8}, we obtain the final identity of the lemma.
	\end{proof}

	
	\section{Umbilical hypersurfaces in products of complex space forms}\label{chapter3}
	
	Niebergall and Ryan~\cite{Niegerball-Ryan} proved that there exist no totally umbilical hypersurfaces in $\mathbb{CP}^n$ or $\mathbb{CH}^n$, a fact first established by Tashiro and Tachibana~\cite{Tashiro-Tachibana} in 1963. Furthermore, they demonstrated that the shape operator cannot be parallel:
	
	\begin{theorem}[Tashiro--Tachibana \& Niebergall--Ryan]\label{t4}
		Let $M$ be a hypersurface in a complex space form with constant holomorphic sectional curvature $4c \neq 0$. Then, the shape operator $A$ can neither be parallel nor totally umbilical; that is, the condition \(A=\lambda I\) cannot hold for any \(\lambda \in C^{\infty}(M)\). Consequently, a complex space form with nonzero constant holomorphic sectional curvature admits no hypersurfaces with a parallel shape operator, nor any totally umbilical hypersurfaces.
	\end{theorem}
	
	In what follows, we present results for the product of two complex space forms, which may be regarded as analogues of the Tashiro--Tachibana and Niebergall--Ryan theorems. The first result demonstrates that there are circumstances under which the first part of Theorem~\ref{t4} remains valid in products of complex space forms.
	
	\begin{proposition}
		\label{prop-obstrucao}
		Let $M$ be a real hypersurface in $\overline{M}=\mathbb{CQ}_1 \times \mathbb{CQ}_2$, where $\mathbb{CQ}_i$ are complex space forms with $c_1\neq 0$ or $c_2\neq 0$. Suppose that $\{V \neq 0\} \subset M$ is nonempty. Then $\nabla A \neq 0$ on $\{V \neq 0\}$.
	\end{proposition}
	
	\begin{proof}
		We argue by contradiction. Assume that $\nabla A = 0$ on $\{V \neq 0\}$. We evaluate the Codazzi equation with $Y = W$. Then
		\[
		(Y \wedge X)V = (W \wedge X)V = \langle X, V \rangle W,
		\]
		and hence
		\[
		L_i\big((Y \wedge X)V\big) = \langle X, V \rangle L_i W.
		\]
		
		Moreover,
		\[
		\langle (Y \wedge X)V, L_i W \rangle
		= (1 + \varepsilon_i h)\langle X, V \rangle,
		\]
		since
		\[
		\langle L_i W, W \rangle
		= \langle W, (I + \varepsilon_i f)W \rangle
		= 1 + \varepsilon_i h.
		\]
		
		Using the properties of $\varphi$, $f$, and $W$, we obtain
		\[
		L_i \varphi L_i W = (\varepsilon_i - h)V.
		\]
		
		Furthermore,
		\[
		(Y \wedge X)L_i W
		= \langle X, L_i W \rangle W - (1 + \varepsilon_i h)X,
		\]
		and therefore
		\[
		L_i \varphi L_i\big((Y \wedge X)L_i W\big)
		= \langle X, L_i W \rangle(\varepsilon_i - h)V
		- (1 + \varepsilon_i h)L_i \varphi L_i X.
		\]
		
		Consequently,
		\begin{equation}\label{eq20}
			d_\nabla A(W,X)
			= \sum_{i=1}^2 \frac{c_i}{2}
			\Big[(7\varepsilon_i + h)\langle X, V \rangle L_i W
			+ \langle X, L_i W \rangle(\varepsilon_i - h)V
			- (1 + \varepsilon_i h)L_i \varphi L_i X\Big].
		\end{equation}
		
		Taking $X = V$ and using
		\[
		L_i V = (1 - \varepsilon_i h)V,
		\qquad
		\langle V, L_i W \rangle = 0,
		\]
		we obtain
		\[
		d_\nabla A(W,V)
		= \sum_{i=1}^2 4c_i(1 - h^2)\big[(\varepsilon_i + h)W - \varphi V\big].
		\]
		
		Thus,
		\[
		(\nabla_V A)W - (\nabla_W A)V
		= \sum_{i=1}^2 4c_i(1 - h^2)\big[(\varepsilon_i + h)W - \varphi V\big].
		\]
		
		Since $|V|^2 = 1 - h^2 \neq 0$ on $\{V \neq 0\}$ and $\nabla A = 0$, it follows that
		\[
		\sum_{i=1}^2 c_i\big[(\varepsilon_i + h)W - \varphi V\big] = 0.
		\]
		
		As $\{W, V, \varphi V\}$ is orthogonal, we conclude that
		\[
		\sum_{i=1}^2 c_i(\varepsilon_i + h) = 0,
		\qquad
		\sum_{i=1}^2 c_i = 0.
		\]
		
		Hence $c_2 = -c_1$, and substituting into the first equation yields $c_1 = 0$, so $c_1 = c_2 = 0$, a contradiction. Therefore, $\nabla A \neq 0$ on $\{V \neq 0\}$.
	\end{proof}
	
	In the next result, we derive an obstruction to the parallelism of the shape operator.
	
	\begin{theorem}\label{t1}
		Let $M$ be a real hypersurface in $\overline{M}=\mathbb{CQ}_1 \times \mathbb{CQ}_2$, where $\mathbb{CQ}_i$ are complex space forms with $c_1\neq 0$ or $c_2\neq 0$. If $M$ is not $F$-invariant, then the shape operator $A$ cannot be parallel. In particular, if $M$ does not admit an almost product Riemannian structure, then $A$ is not parallel.
	\end{theorem}
	
	\begin{proof}
		By Proposition~\ref{f-invariancia}, the condition that $M$ is not $F$-invariant implies that $V \neq 0$ at some point of $M$. Hence, the subset $\{V \neq 0\} \subset M$ is nonempty, and Proposition~\ref{prop-obstrucao} yields $\nabla A \not\equiv 0$ on $M$.
		
		The last assertion follows again from Proposition~\ref{f-invariancia}, since the absence of an almost product Riemannian structure implies that $M$ is not $F$-invariant.
	\end{proof}
	
	In the next lemma, we derive a formula for the gradient of the mean curvature of a totally umbilical real hypersurface, which will be essential in the proofs of the main results.
	
	\begin{lemma}\label{l2}
		Let $M$ be a totally umbilical real hypersurface of $\overline{M}=\mathbb{CQ}_1 \times \mathbb{CQ}_2$. Then
		\[
		\nabla H
		=4\left(\sum_{i=1}^2 c_i(\varepsilon_i+h)\right)V.
		\]
	\end{lemma}
	
	\begin{proof}
		Let $M^{2n-1}$ be a real hypersurface of $\overline{M}^{2n}$ with unit normal vector field $\nu$. Consider the local orthonormal frame $\{W, e_j, \varphi e_j\}_{j=1}^{n-1}$ on $M$ introduced in Lemma~\ref{l1}, so that $\{\nu, W, e_j, \varphi e_j\}_{j=1}^{n-1}$ is a local frame on $\overline{M}$.
		
		From the Codazzi equation, assuming $A=\lambda I$, we obtain
		\[
		\begin{aligned}
			d_\nabla A(W,e_j)
			&=(\nabla_{e_j}A)W-(\nabla_WA)e_j
			=(e_j\lambda)W-(W\lambda)e_j,\\
			d_\nabla A(W,\varphi e_j)
			&=(\nabla_{\varphi e_j}A)W-(\nabla_WA)\varphi e_j
			=(\varphi e_j\lambda)W-(W\lambda)\varphi e_j.
		\end{aligned}
		\]
		
		Comparing these expressions with those obtained in Lemma~\ref{l1}, we deduce
		\begin{align}
			e_j\lambda&=\sum_{i=1}^{2}4c_i(\varepsilon_i+h)\langle V,e_j\rangle,\label{eq3}\\
			(W\lambda)\delta_{jk}&=\sum_{i=1}^{2}\frac{c_i}{2}
			\Big[(7+\varepsilon_i h)\langle e_j,V\rangle\langle V,\varphi e_k\rangle
			+(1-\varepsilon_i h)\langle \varphi e_j,V\rangle\langle V,e_k\rangle \nonumber\\
			&\hspace{2cm}-(1+\varepsilon_i h)\langle L_i\varphi L_i e_j,e_k\rangle\Big],\label{eq4}\\
			(\varphi e_j)\lambda&=\sum_{i=1}^{2}4c_i(\varepsilon_i+h)\langle V,\varphi e_j\rangle,\label{eq5}\\
			(W\lambda)\delta_{jk}&=\sum_{i=1}^{2}\frac{c_i}{2}
			\Big[-(7+\varepsilon_i h)\langle \varphi e_j,V\rangle\langle V,e_k\rangle
			-(1-\varepsilon_i h)\langle e_j,V\rangle\langle V,\varphi e_k\rangle \nonumber\\
			&\hspace{2cm}+(1+\varepsilon_i h)\langle \varphi L_i\varphi L_i\varphi e_j,e_k\rangle\Big].
			\label{eq6}
		\end{align}
		
		Since $\varphi$ is skew-symmetric and $L_i$ is symmetric, we have
		\[
		\langle L_i\varphi L_i e_j,e_j\rangle = 0,
		\qquad
		\langle \varphi L_i\varphi L_i\varphi e_j,e_j\rangle = 0.
		\]
		
		Setting $k=j$ in~\eqref{eq4} and~\eqref{eq6} and summing the resulting equations, we obtain $W\lambda=0$.
		
		Moreover,
		\[
		\left(\sum_{i=1}^{2}4c_i(\varepsilon_i+h)\right)V
		=\left(\sum_{i=1}^{2}4c_i(\varepsilon_i+h)\right)
		\sum_{j=1}^{n-1}
		\big(\langle V,e_j\rangle e_j+\langle V,\varphi e_j\rangle\varphi e_j\big),
		\]
		since $\langle V,W\rangle=0$. Using~\eqref{eq3} and~\eqref{eq5}, we conclude that
		\[
		\nabla\lambda
		=\sum_{j=1}^{n-1}\big((e_j\lambda)e_j+(\varphi e_j\lambda)\varphi e_j\big)
		=\left(\sum_{i=1}^{2}4c_i(\varepsilon_i+h)\right)V.
		\]
		
		Since $M$ is totally umbilical, $\lambda=H$, and the result follows.
	\end{proof}
	
	We now characterize totally umbilical hypersurfaces in products of complex space forms.
	
	\begin{theorem}\label{t3}
		Let $M$ be an oriented totally umbilical real hypersurface of
		$\overline{M}=\mathbb{CQ}_1 \times \mathbb{CQ}_2$. If $M$ is $F$-invariant, then $M$ has constant mean curvature. In particular, $M$ is either totally geodesic or an extrinsic hypersphere, that is, it is a totally umbilical hypersurface with constant nonzero mean
		curvature.
	\end{theorem}
	
	\begin{proof}
		Let $M$ be a totally umbilical hypersurface of $\mathbb{CQ}_1 \times \mathbb{CQ}_2$. Since $M$ is $F$-invariant, Proposition~\ref{f-invariancia} implies that $V \equiv 0$. By Lemma~\ref{l2}, it follows that $\nabla H = 0$, hence $H$ is constant. If $H \equiv 0$, then $M$ is totally geodesic; otherwise, $M$ is an extrinsic hypersphere.
	\end{proof}
	
	We now consider the complementary case, where $M$ is not $F$-invariant.
	
	\begin{theorem}\label{t2}
		Let $M$ be an oriented totally umbilical real hypersurface of
		$\overline{M}=\mathbb{CQ}_1 \times \mathbb{CQ}_2$, where $\mathbb{CQ}_i$ are complex space forms with $c_1\neq 0$ or $c_2\neq 0$. If $M$ is not $F$-invariant, then $M$ does not have constant mean curvature. In particular, $M$ is neither totally geodesic nor an extrinsic hypersphere.
	\end{theorem}
	
	\begin{proof}
		Let $M$ be totally umbilical, so that $A = H I$. If $M$ is not $F$-invariant, then Proposition~\ref{f-invariancia} implies that $\{V \neq 0\}$ is nonempty. If $H$ were constant, then $\nabla A \equiv 0$, contradicting Theorem~\ref{t1}. Hence $H$ is not constant, and the conclusion follows.
	\end{proof}
	
	\begin{corollary}
		Let $M$ be an oriented real hypersurface of
		$\overline{M}=\mathbb{CQ}_1 \times \mathbb{CQ}_2$ that is not $F$-invariant, where $\mathbb{CQ}_i$ are complex space forms with $c_1\neq 0$ or $c_2\neq 0$. If $M$ has constant mean curvature, then $M$ is not umbilical.
	\end{corollary}
	
	\begin{proof}
		This follows from the contrapositive of Theorem~\ref{t2}.
	\end{proof}
	
	\bibliographystyle{amsplain}
	\bibliography{references}
\end{document}